\documentclass[a4paper,12pt]{article}
\usepackage[cp1251]{inputenc}
\usepackage[russian, ukrainian,english]{babel}
\usepackage{amsmath, amsthm, amsfonts, amssymb}
\usepackage{graphicx}

\theoremstyle{plain}
\newtheorem{theorem}{Theorem}[section]

\newtheorem{lemma}{Lemma}[section]
\newtheorem{corollary}{Corollary}[section]
\newtheorem{remark}{Remark}[section]
\newtheorem{definition}{Definition}[section]

\begin{document}
	
	\begin{center}
		{\bf Duality for coalescing stochastic flows on the real line}
		
		\vskip20pt

		Georgii V. Riabov
		
		\vskip20pt	
		
		Institute of Mathematics, NAS of Ukraine

			\end{center}
		
		\vskip20pt

		{\small {\bf Abstract.}  For a class of coalescing stochastic flows on the real line the existence of dual flows is proved. A stochastic flow and its dual are constructed as a forward and backward perfect cocycles over the same metric dynamical system. The metric dynamical system itself is defined on a new state space for coalescing flows.  General results are applied to Arratia flows with drift.}

\section{Introduction}

In the present work we study duality for coalescing stochastic flows on the real line from the persepective of random dynamical systems. By a flow on $\mathbb{R}$ we understand a family $\{\psi_{s,t}:-\infty<s\leq t<\infty\}$ of mappings of $\mathbb{R},$ $\psi_{s,t}:\mathbb{R}\to\mathbb{R},$  that are related by the evolutionary property:

for all $r\leq s\leq t,$ $x\in\mathbb{R},$
\begin{equation}
\label{eq15_1}
\psi_{s,t}(\psi_{r,s}(x))=\psi_{r,t}(x) \mbox{ and } \psi_{s,s}(x)=x.
\end{equation}
Respectively, a stochastic flow on $\mathbb{R}$ is a family $\{\psi_{s,t}:-\infty<s\leq t<\infty\}$ of random mappings of $\mathbb{R},$ $\psi_{s,t}:\Omega\times\mathbb{R}\to\mathbb{R},$ that satisfy the evolutionary property \eqref{eq15_1} without exceptions in $\omega,$ are homogeneous (i.e. distributions of random vectors $(\psi_{s,t}(x_1),\ldots,\psi_{s,t}(x_n))$ and $(\psi_{s+h,t+h}(x_1),\ldots,\psi_{s+h,t+h}(x_n))$ coincide), and possess independent increments (i.e. for all $t_1\leq t_2\leq \ldots\leq t_n$ random mappings $\psi_{t_1,t_2},\ldots,\psi_{t_{n-1},t_n}$ are independent, see section 2 for precise definitions).  We consider only flows with continuous trajectories, i.e. for each $(s,x)\in\mathbb{R}^2$ the function $t\to \psi_{s,t}(x)$ is continuous on $[s,\infty).$ 

Given a flow $\psi$ on $\mathbb{R},$ its dual is a flow on $\mathbb{R}$ that evolves backwards in time never crossing trajectories of $\psi.$ Formally, a backward flow on $\mathbb{R}$ is a family $\{\tilde{\psi}_{t,s}:-\infty<s\leq t<\infty\}$ of mappings of $\mathbb{R},$ $\tilde{\psi}_{t,s}:\mathbb{R}\to\mathbb{R},$ that are related by the backward evolutionary property:

for all $r\leq s\leq t,$ $y\in\mathbb{R},$
\begin{equation}
\label{eq15_2}
\tilde{\psi}_{s,r}(\tilde{\psi}_{t,s}(y))=\tilde{\psi}_{t,r}(y) \mbox{ and } \tilde{\psi}_{s,s}(y)=y.
\end{equation}
Again, we assume that  functions $s\to \tilde{\psi}_{t,s}(y)$ are continuous on $(-\infty,t]$ for all $(t,y)\in\mathbb{R}^2.$

\begin{definition}
\label{def15_1} \cite{Arratia1, Arratia2} Backward flow $\tilde{\psi}$ is dual to the flow $\psi,$ if for all $s\leq t,$ $x,y\in\mathbb{R}$
\begin{equation}
\label{eq15_3}
(\psi_{s,t}(x)-y)(x-\tilde{\psi}_{t,s}(y))\geq 0.
\end{equation}
\end{definition}
If $\psi$ is a stochastic flow on $\mathbb{R},$ then a dual flow is a backward stochastic flow $\tilde{\psi}$ that satisfies \eqref{eq15_3} without exceptions in $\omega$ (see section 2 for the rigorous definition of a backward stochastic flow). 

Following \cite{LeJanRaimond},  a stochastic flow $\psi$ is called coalescing if for some distinct $x,y\in\mathbb{R}$
$$
\mathbb{P}(\exists t>0: \ \psi_{s,t}(x)=\psi_{s,t}(y))>0.
$$
The flows we study possess stronger property: with probability $1$ for all $s<t$ images $\psi_{s,t}(\mathbb{R})$ are locally finite subsets of $\mathbb{R}.$  In other words, for $s<t$ mappings $x\to\psi_{s,t}(x)$ are random step functions. This constrasts the well-known case of stochastic flows of homeomorphisms treated in \cite{Kunita}. For example, consider an It\^o's stochastic differential equation
\begin{equation}
\label{eq15_4}
dX(t)=a(X(t))dt+b(X(t))dw(t),
\end{equation}
where $w$ is a Wiener process and coefficients $a,b$ are globally Lipschitz. The equation \eqref{eq15_4} can be solved simultaneously for all starting points $(s,x)\in\mathbb{R}^2$ giving rise to a stochastic flow $\{\psi_{s,t}:-\infty<s\leq t<\infty\}$ of homeomorphisms of $\mathbb{R}$ \cite[Section 4]{Kunita2}. In this case the dual flow is $\tilde{\psi}_{t,s}=\psi^{-1}_{s,t}.$ Its properties are described in detail in \cite[Ch. 4]{Kunita}. 

One of the most known and studied examples of a coalescing stochastic flow on $\mathbb{R}$ is the Arratia flow. It describes a motion of a continuum family of  Wiener processes that start from every time-space point $(s,x)\in\mathbb{R}^2,$ move independently before meeting and coalesce at a meeting time. In \cite{Arratia2} the existence of a corresponding stochastic flow $\{\psi_{s,t}:-\infty<s\leq t<\infty\}$ was proved (see \cite{Darling, TW, LeJanRaimond, FINR, SSS,  NT, BGS, Riabov} for a number of modifications and generalizations). One consequence of independent motion before meeting time is that with probability 1 for any $s<t$ and $a<b$ the set $\psi_{s,t}([a,b])$ is finite. Duality for the Arratia flow was also developed in \cite{Arratia2}. Mappings $x\to\psi_{s,t}(\omega,x)$  are not invertible, but there are two natural candidates for a dual flow:

\begin{itemize}
\item a family of right-continuous generalized inverses
$$
v^+_{t,s}(y)=\inf\{x\in\mathbb{R}:\psi_{s,t}(x)>y\};
$$

\item a family of left-continuous generalized inverses
$$
v^-_{t,s}(y)=\inf\{x\in\mathbb{R}:\psi_{s,t}(x)\geq y\}.
$$

\end{itemize}
For fixed $t\in\mathbb{R}$ and $y_1,\ldots,y_n\in\mathbb{R}$ processes $s\to (v^+_{t,s}(y_1),\ldots,v^+_{t,s}(y_n))$  and $s\to (v^-_{t,s}(y_1),\ldots,v^-_{t,s}(y_n))$ are coalescing Wiener processes that move (backwards) independently before the meeting time. However, neither $v^+$ nor $v^-$ is a backward stochastic flow - with probability 1 the property \eqref{eq15_2} fails for each fo them \cite{Arratia2}. It was suggested in \cite{Arratia2} that a proper choice between $v^+$ and $v^-$ gives rise to a backward flow dual to $\psi.$ We generalize and prove this statement in section 3. Thus, dual flow to the Arratia flow exists and is the Arratia flow. Despite the flow property for duals of coalescing stochastic flows on $\mathbb{R}$ wasn't study in general,  the generalize inverses $v^+$ and $v^-$ of the Arratia flow were successfully applied in \cite{Arratia2,Harris,TW,FINR,DK, DKG,DRS}. In this paper we fill the gap with the evolutionary property of dual flows for a class of coalescing stochastic flows on $\mathbb{R}$ (see sections 2 and 4 for the conditions we impose on a stochastic  flow).

Another novelty of our work is the description of a dual flow as a random dynamical system in the sense of \cite{Arnold}. Consider a probability space $(\mathbb{F},\mathcal{A},\mu)$ equipped with a measurable group $(\theta_h)_{h\in\mathbb{R}}$ of measure-preserving transformations of $\mathbb{F},$ and a perfect cocylce $\varphi$ over $\theta$ -- a measurable mapping 
$\varphi:[0,\infty)\times\Omega\times\mathbb{R}\to\mathbb{R}$ such that for all $s,t\geq 0,$ $\omega\in\Omega,$ $x\in\mathbb{R},$
\begin{equation}
\label{eq15_5}
\varphi(t+s,\omega,x)=\varphi(t,\theta_s\omega,\varphi(s,\omega,x)) \mbox{ and } \varphi(0,\omega,x)=x.
\end{equation}
The perfect cocycle property \eqref{eq15_5} immediately implies that for all $\omega\in\Omega$
$$
\psi_{s,t}(\omega,x)=\varphi(t-s,\theta_s\omega,x)
$$
is a flow of mappings of $\mathbb{R}.$ In \cite{Riabov} general conditions on a coalescing stochastic flow $\psi$ were formulated under which $\psi$ is generated by a random dynamical system in the described way. The representation of a flow via a random dynamical system endows a flow with a richer structure that allows to develop ergodic theory \cite{Arnold}. For example, in \cite{DRS} random dynamical systems were applied to study stationary points and invariant measures for Arratia flows with drift.   It is a natural question then whether a dual flow is generated by a random dynamical system.

In our main result (theorem \ref{thm1}) we give conditions on a coalescing stochastic flow $\psi$ under which both the flow  and its dual are generated by random dynamical systems. Namely, starting from finite-point motions of $\psi$ on a certain probability space $(\mathbb{F},\mathcal{A},\mu)$ with a measurable group of measure preserving transformations $(\theta_h)_{h\in\mathbb{R}}$ we construct a perfect cocycle $\varphi$ and a backward perfect cocycle $\tilde{\varphi}$ such that $\psi_{s,t}(\omega,x)=\varphi(t-s,\theta_s\omega,x)$ is a stochastic flow on $\mathbb{R}$ with prescribed finite-point motions and $\tilde{\psi}_{t,s}(\omega,x)=\tilde{\varphi}(t-s,\theta_s\omega,x)$ is a backward stochastic flow dual to $\psi.$ 

In section 2 we collect all the necessary definitions and formulate the main theorem. Section 3 is devoted to the construction of a measurable space $(\mathbb{F},\mathcal{A})$ together with a measurable group of transformations $(\theta_h)_{h\in\mathbb{R}}$ and two perfect cocycles $\varphi$ and $\tilde{\varphi}$ that generate dual flows. The space $\mathbb{F}$ is actually a specific space of flows $\omega=\{\omega_{s,t}:-\infty<s\leq t<\infty\},$ $\theta_h$ being a time shift: $(\theta_h\omega)_{s,t}=\omega_{s+h,t+h}.$ The dual flow is constructed as a measurable functional on $\mathbb{F}$ which can be of independent interest. In section 4 we define a probability measure $\mu$  on $(\mathbb{F},\mathcal{A})$ that is $\theta_h-$invariant and is such that on the space $(\mathbb{F},\mathcal{A},\mu)$ the canonical flow $\psi_{s,t}(\omega,x)=\omega_{s,t}(x)$ is the needed stochastic flow. By construction, the flow  $\psi$ is generated by a random dynamical system $\varphi$ and the flow $\tilde{\psi}$ is generated by a backward random dynamical system $\tilde{\varphi}.$ The distribution of $\tilde{\psi}$ is described in section 5. We prove that $\tilde{\psi}$ is a backward stochastic flow on $\mathbb{R}$ and characterize its finite-point motions.  Finally, in section 6 we apply the theory to the Arratia flows with drift. We show that the dual flow exists and is the Arratia flow with drift. This recovers and strengthes results of \cite{TW, DRS}.

\section{Preliminaries and the main result}

We will consider sets $\mathbb{R}_+= [0,\infty),$ $\mathcal{H}=\{(s,t)\in\mathbb{R}^2:s\leq t\}.$  The complement of the set $A$ is denoted by $A^c.$ Integration with respect to the probability measure $\mu$ will be denoted by $\mathbb{E}_\mu.$ The Borel $\sigma-$field on a  metric spaces $X$ will be denoted by $\mathcal{B}(X).$ By $C_0(\mathbb{R}^n)$ we denote the space of continuous functions $f:\mathbb{R}\to\mathbb{R}$ such that $\lim_{|x|\to\infty}f(x)=0.$

Below we formulate few important results on stochastic flows following mainly  \cite{LeJanRaimond}. 

The distribution of a stochastic flow is determined by its finite-point motions. Let $\{P^{(n)}:n\geq 1\}$ be a sequence of transition probabilities satisfying following three conditions.

\begin{itemize}
\item {\bf (TP1)} $P^{(n)}=\{P^{(n)}_t:t\geq 0\}$ is a Feller transition probability on $(\mathbb{R}^n,\mathcal{B}(\mathbb{R}^n)).$ 

\item {\bf (TP2)} Given $\{i_1,\ldots,i_k\}\subset \{1,\ldots,n\}$ let $\pi_{i_1,\ldots,i_k}:\mathbb{R}^n\to\mathbb{R}^k$ be a projection, $\pi_{i_1,\ldots,i_k}(x)=(x_{i_1},\ldots,x_{i_k}).$ Then for all $t\geq 0,$ $x\in\mathbb{R}^n$ and $C\in\mathcal{B}(\mathbb{R}^k),$
$$
P^{(n)}_t(x,\pi^{-1}_{i_1,\ldots,i_k}C)=P^{(k)}_t(\pi_{i_1,\ldots,i_k}x,C).
$$

\item {\bf (TP3)} Let $\Delta=\{(y,y):y\in\mathbb{R}\}$  be a diagonal in $\mathbb{R}^2.$ Then for all $t\geq 0$ and $x\in \Delta$
$$
P^{(2)}_t(x,\Delta)=1.
$$

\end{itemize}

When conditions {\bf (TP1)-(TP3)} are satisfied the sequence $\{P^{(n)}:n\geq 1\}$ will be called a compatible sequence of coalescing Feller transition probabilities on $\mathbb{R}$. It will describe finite-point motions of a stochastic flow. We use the definition of a stochastic flow from \cite{Riabov}. In slightly different forms it appeared in \cite{Darling, LeJanRaimond}.

\begin{definition}
\label{def15_2} Let $\{P^{(n)}:n\geq 1\}$ be a compatible sequence of coalescing Feller transition probabilities on $\mathbb{R}$. A stochastic flow on $\mathbb{R}$ with finite-point motions determined by $\{P^{(n)}:n\geq 1\}$ is a family $\{\psi_{s,t}(x):-\infty<s\leq t<\infty,x\in\mathbb{R}\}$ of random variables (defined on a probability space $(\Omega,\mathcal{F},\mathbb{P})$), such that 

\begin{itemize}
\item {\bf (SF1)} the mapping $(s,t,\omega,x)\to \psi_{s,t}(\omega,x)$ is $\mathcal{B}(\mathcal{H})\otimes \mathcal{F}\otimes \mathcal{B}(\mathbb{R})/\mathcal{B}(\mathbb{R})$-measurable;

\item {\bf (SF2)} for all $s\leq t,$ $x\in\mathbb{R},$ $\omega\in\Omega,$ 
$$
\psi_{s,t}(\omega,\psi_{r,s}(\omega,x))=\psi_{r,t}(\omega,x) \mbox{ and } \psi_{s,s}(\omega,x)=x;
$$

\item {\bf (SF3)} given $s\in\mathbb{R}$ and a random vector $\xi=(\xi_1,\ldots,\xi_n)$ measurable with respect to the ``past'' $\sigma-$field  $\mathcal{F}^\psi_{-\infty,s}=\sigma(\{\psi_{u,v}(x):u\leq v\leq s, x\in\mathbb{R}\}),$ for all $t\geq s$ and $B\in\mathcal{B}(\mathbb{R}^n)$
$$
\mathbb{P}((\psi_{s,t}(\xi_1),\ldots,\psi_{s,t}(\xi_n))\in B|\mathcal{F}^\psi_{-\infty,s})=P^{(n)}_{t-s}(\xi,B) \mbox{ a.s.}
$$

\end{itemize}

\end{definition}

\begin{remark} 
\label{rem15_1} The property {\bf (SF3)} implies homogeneity and independence of increments: for fixed $x\in \mathbb{R}^n$ the law of $(\psi_{s,t}(x_1),\ldots,\psi_{s,t}(x_n))$ is $P^{(n)}_{t-s}(x,\cdot);$ for $t_1\leq t_2\leq \ldots\leq t_n$ mappings $\psi_{t_1,t_2},\ldots,\psi_{t_{n-1},t_n}$ are independent. See \cite{Riabov} for an example showing that {\bf (SF3)}  is stronger than these two properties. Also, in \cite{Riabov} it is proved that finite-dimensional distributions of the flow $\psi$ are uniquely determined by properties {\bf (SF1)-(SF3)}.

\end{remark}

\begin{definition}
\label{def15_3} Let $\{P^{(n)}:n\geq 1\}$ be a compatible sequence of coalescing Feller transition probabilities on $\mathbb{R}$. A family $\{\tilde{\psi}_{t,s}(y):-\infty<s\leq t<\infty,y\in\mathbb{R}\}$ of random variables is a backward stochastic flow on $\mathbb{R}$ with finite-point motions determined by $\{P^{(n)}:n\geq 1\}$ if  the family $\{\psi_{s,t}(x):-\infty<s\leq t<\infty\}$ defined by $\psi_{s,t}(x)=\tilde{\psi}_{-s,-t}(x)$ is a stochastic flow on $\mathbb{R}$ with finite-point motions determined by $\{P^{(n)}:n\geq 1\}.$

\end{definition}

If $\psi$ is a stochastic flow on $\mathbb{R},$ then for every $\omega$ the family of mappings $\{\psi_{s,t}(\omega,\cdot):-\infty<s\leq t<\infty\}$ is a flow on $\mathbb{R}$, i.e. \eqref{eq15_1} holds.  If $\tilde{\psi}$ is a backward stochastic flow on $\mathbb{R},$ then for every $\omega$ the family of mappings $\{\tilde{\psi}_{t,s}(\omega,\cdot):-\infty<s\leq t<\infty\}$ is a backward flow on $\mathbb{R},$  i.e. \eqref{eq15_2} holds. Assume that a stochastic flow $\psi$ and a backward stochastic flow $\tilde{\psi}$ on $\mathbb{R}$ are defined on a single probability space. We will say that $\tilde{\psi}$ is dual to $\psi,$ if for every $\omega$ the backward flow $\{\tilde{\psi}_{t,s}(\omega,\cdot):-\infty<s\leq t<\infty\}$ is dual to the flow $\{\psi_{s,t}(\omega,\cdot):-\infty<s\leq t<\infty\}$ in the sense of definition \ref{def15_1}.

To construct stochastic flows and their duals we use a convenient framework of random dynamical systems. We briefly recall the main notions and relations with stochastic flows. For an account of the topic we refer to \cite{Arnold}.

\begin{definition}
\label{def15_4} A metric dynamical system is a probability space $(\Omega,\mathcal{F},\mathbb{P})$ equipped with a measurable group of measure preserving transformations $(\theta_h)_{h\in\mathbb{R}}.$ That is, the mapping
$$
(\omega,h)\to\theta_h\omega
$$
is $\mathcal{F}\otimes \mathcal{B}(\mathbb{R})/\mathcal{F}$-measurable; for all $s,h\in\mathbb{R}$ and $\omega\in\Omega,$
$$
\theta_{s+h}\omega=\theta_s\theta_h\omega \mbox{ and } \theta_0\omega=\omega;
$$
for all $h\in\mathbb{R},$ $\mathbb{P}\circ\theta^{-1}_h=\mathbb{P}.$

\end{definition}

\begin{definition}
\label{def15_5} Let $(\Omega,\mathcal{F},\mathbb{P},(\theta_h)_{h\in\mathbb{R}})$ be a  metric dynamical system. A perfect cocycle over $\theta$ is an $\mathcal{B}(\mathbb{R}_+)\otimes \mathcal{F}\otimes \mathcal{B}(\mathbb{R})/\mathcal{B}(\mathbb{R})$-measurable mapping 
$$
\varphi:\mathbb{R}_+\times\Omega\times\mathbb{R}\to\mathbb{R},
$$
such that for all $s,t\geq 0,$ $x\in\mathbb{R},$ $\omega\in\Omega,$
$$
\varphi(t+s,\omega,x)=\varphi(t,\theta_s\omega,\varphi(s,\omega,x)) \mbox{ and } \varphi(0,\omega,x)=x.
$$
A backward perfect cocycle over $\theta$ is an $\mathcal{B}(\mathbb{R}_+)\otimes \mathcal{F}\otimes \mathcal{B}(\mathbb{R})/\mathcal{B}(\mathbb{R})$-measurable mapping 
$$
\tilde{\varphi}:\mathbb{R}_+\times\Omega\times\mathbb{R}\to\mathbb{R},
$$
such that for all $s,t\geq 0,$ $x\in\mathbb{R},$ $\omega\in\Omega,$
$$
\tilde{\varphi}(t+s,\omega,x)=\tilde{\varphi}(t,\omega,\tilde{\varphi}(s,\theta_t\omega,x)) \mbox{ and } \tilde{\varphi}(0,\omega,x)=x.
$$

\end{definition}

Given a perfect cocycle $\varphi$ over $\theta$ it is immediate that the relation
$$
\psi_{s,t}(\omega,x)=\varphi(t-s,\theta_s\omega,x)
$$
defines $\omega-$wisely a flow of mappings of $\mathbb{R},$ and the mapping $(s,t,\omega,x)\to \psi_{s,t}(\omega,x)$ is jointly measurable. Thus, to prove that $\psi$ is a stochastic flow on $\mathbb{R}$ one has to check the property {\bf (SF3)} with some compatible sequence of coalescing Feller transition pro\-ba\-bi\-li\-ties. Same observation is applicable to the backward cocycle $\tilde{\varphi}$ and a backward flow of mappings 
$$
\tilde{\psi}_{t,s}(\omega,x)=\tilde{\varphi}(t-s,\theta_s\omega,x).
$$
 
To formulate the result we add more assumptions on finite-point motions of the stochastic flow $\psi.$ The assumptions contain those of \cite[Th. 1.1]{Riabov}, with the  new assumption {\bf (TP7)} on meeting times for three-point motions. We prove that under this new assumption it is possible to construct both the flow $\psi$ and its dual.

\begin{itemize}
\item {\bf (TP4)} For all $t>0,$ $x,y\in\mathbb{R}$
$$
P^{(1)}_t(x,\{y\})=0.
$$

\item {\bf (TP5)} For all real $a<b$ and $\varepsilon>0$
$$
\lim_{t\to 0}t^{-1}\sup_{x\in[a,b]}P^{(1)}_t(x,(x-\varepsilon,x+\varepsilon)^c)=0
$$

\begin{remark}
\label{rem16_1}
Under the condition {\bf (TP5)} the Feller process corresponding to $P^{(1)}$ has a.s. continuous trajectories \cite[Ch. 4, Prop. 2.9]{EK}. We denote by $\mathbb{P}^{(n)}_x$ the distribution in $C([0,\infty),\mathbb{R}^n)$ of a Feller process with transition probability $P^{(n)}$ and a starting point $x.$ The canonical process on $C([0,\infty),\mathbb{R}^n)$ will be denoted by $\{X^{(n)}(t)=(X^{(n)}_1(t),\ldots,X^{(n)}_n(t)):t\geq 0\},$ so that for all $0<t_1<\ldots<t_k,$ $x\in\mathbb{R}^n$ and $A_1,\ldots,A_k\in\mathcal{B}(\mathbb{R}^n)$
$$
\mathbb{P}^{(n)}_x (X^{(n)}(t_1)\in A_1,\ldots, X^{(n)}(t_k)\in A_k)=
$$
$$
=\int_{A_1}\ldots \int_{A_{k-1}}P^{(n)}_{t_k-t_{k-1}}(u_{k-1},A_k)P^{(n)}_{t_{k-1}-t_{k-2}}(u_{k-2},du_{k-1})\ldots P^{(n)}_{t_1}(x,du_1).
$$
\end{remark}

\item {\bf (TP6)} Given reals $a<b$ and $t>0$ there exists an increasing continuous function $m_{a,b,t}:\mathbb{R}\to\mathbb{R}$ such that for all $x_1,x_2$ with  $a\leq x_1<x_2\leq b,$
$$
\mathbb{P}^{(2)}_{(x_1,x_2)}(\forall s\in [0,t] \ a\leq X^{(2)}_1(s)<X^{(2)}_2(s)\leq b)\leq m_{a,b,t}(x_2)-m_{a,b,t}(x_1).
$$

\item {\bf (TP7)} Given reals $a<b$ and $t>0$ there exists a positive function $f_{a,b,t}:(0,\infty)\to(0,\infty)$ such that for all $x_1,x_2,x_3$ with  $a\leq x_1<x_2<x_3\leq b,$
$$
\mathbb{P}^{(3)}_{(x_1,x_2,x_3)}(\forall s\in [0,t] \ a\leq X^{(3)}_1(s)<X^{(3)}_2(s)<X^{(3)}_3(s)\leq b)\leq f_{a,b,t}(x_3-x_1).$$

\end{itemize}

Let 
\begin{equation}
\label{eq08_1}
w_{a,b}(\varepsilon,\delta)=\inf\{t>0: \sup_{x\in[a,b]}P^{(1)}_t(x,(x-\varepsilon,x+\varepsilon)^c)\geq \delta t\}.
\end{equation}
The following is the main theorem of the paper.

\begin{theorem}
\label{thm1} Let $\{P^{(n)}:n\geq 1\}$ be a compatible sequence of coalescing Feller  transition probabilities on $\mathbb{R}$ satisfying conditions {\bf (TP1)-(TP7)}.
 Assume that for any reals $a<b$ and $t>0$
\begin{equation}
\label{eq21_3}
 \liminf_{\varepsilon,\delta\to 0}\frac{f_{a,b,t}(8\varepsilon)}{w_{a,b}(\varepsilon,\delta)}=0.
\end{equation}
Then there exists a metric dynamical system $(\mathbb{F},\mathcal{A},\mu,(\theta_h)_{h\in\mathbb{R}}),$ a perfect cocycle $\varphi$ over $\theta$ and a backward perfect cocycle $\tilde{\varphi}$ over $\theta,$ such that

\begin{enumerate}
\item the flow $\psi_{s,t}(\omega,x)=\varphi(t-s,\theta_s\omega,x)$ is a stochastic flow on $\mathbb{R}$ with finite-point motions determined by $\{P^{(n)}:n\geq 1\};$

\item the backward flow $\tilde{\psi}_{t,s}(\omega,x)=\tilde{\varphi}(t-s,\theta_s\omega,x)$ is a backward stochastic flow on $\mathbb{R};$

\item the backward stochastic flow $\tilde{\psi}$ is dual to the stochastic flow $\psi.$

\end{enumerate}

Moreover, the finite-point motions  of $\tilde{\psi}$ are determined by a sequence $\{\tilde{P}^{(n)}:n\geq 1\}$ which is a unique compatible  sequence of coalescing Feller transition probabilities on $\mathbb{R}$ that satisfy the duality relation
$$
\tilde{P}^{(n)}(y,(x_1,x_2)\times(x_2,x_3)\times\ldots\times (x_n,\infty))=
$$
$$
=P^{(n)}(x,(-\infty,y_1)\times(y_1,y_2)\times\ldots\times (y_{n-1},y_n))
$$
for all $n\geq 1,$ $t\geq 0$ and  $x,y\in\mathbb{R}^n$ such that $x_1<y_1<x_2<y_2<\ldots<x_n<y_n.$

\end{theorem}

Next three sections are devoted to the proof. In section 3 we construct the measurable space of flows $(\mathbb{F},\mathcal{A})$ together with a group of shifts $(\theta_h)_{h\in\mathbb{R}}$ and two perfect cocycles $\varphi$ and $\tilde{\varphi}$ that give rise to dual flows $\psi$ and $\tilde{\psi}$. In section 4 we define a measure $\mu$ on $(\mathbb{F},\mathcal{A})$ that makes $(\mathbb{F},\mathcal{A},\mu,(\theta_h)_{h\in\mathbb{R}})$ a metric dynamical system and such that $\psi$ becomes a stochastic flow with prescribed finite-point motions. In section 5 we prove that under the measure $\mu,$ $\tilde{\psi}$ is a backward stochastic flow and characterize its finite-point motions. The construction is applied to the Arratia flow with drift in section 6.

\section{Space of flows $\mathbb{F}$}

In this section we construct a space $\mathbb{F}$ of coalescing flows on $\mathbb{R}$ that carries a metric dynamical system described in the theorem \ref{thm1}. A generic element $\omega\in\mathbb{F}$ is a flow of mappings of $\mathbb{R},$ $\omega=\{\omega_{s,t}:-\infty<s\leq t<\infty\}$  that satisfies properties {\bf (C1)-(C5)} below. We equip $\mathbb{F}$ with a cylindrical $\sigma-$field $\mathcal{A}$ and define a group of shifts $(\theta_h)_{h\in\mathbb{R}},$  a perfect cocycle $\varphi$ and a backward perfect cocycle $\tilde{\varphi}$ over $\theta$ in such a way that mappings 
$$
\psi_{s,t}(\omega,x)=\varphi(t-s,\theta_s\omega,x)
$$
and 
$$
\tilde{\psi}_{t,s}(\omega,x)=\tilde{\varphi}(t-s,\theta_s\omega,x)
$$
are a pair of forward and backward flows in duality (in the sense of the definition \ref{def15_1}).

Let $C_x([s,\infty))$ be the space of all continuous functions $f:[s,\infty)\to\mathbb{R}$ with $f(s)=x.$  We consider the product $\prod_{(s,x)\in\mathbb{R}^2}C_x([s,\infty)).$ An element $\omega\in\prod_{(s,x)\in\mathbb{R}^2}C_x([s,\infty))$ is a collection of functions $t\to\omega_{s,t}(x),$ $t\in[s,\infty),$ indexed by all time-space points $(s,x)\in\mathbb{R}^2.$  We will denote $\omega=\{\omega_{s,t}:-\infty<s\leq t<\infty\}.$

\begin{definition}
\label{def_flow}
The space $\mathbb{F}$ of coalescing stochastic flows is a set of all flows $\omega\in\prod_{(s,x)\in\mathbb{R}^2}C_x([s,\infty))$ that satisfy following five conditions.

\begin{itemize}
\item {\bf (C1)} For all $r\leq s\leq t,$ $x\in\mathbb{R}$ 
$$
\omega_{s,t}(\omega_{r,s}(x))=\omega_{r,t}(x).
$$

\item {\bf (C2)} For all $s<t$ the image $\omega_{s,t}(\mathbb{R})$ is a locally finite subset of $\mathbb{R}$ with 
$$
\sup \omega_{s,t}(\mathbb{R})=\infty, \inf \omega_{s,t}(\mathbb{R})=-\infty.
$$

\item {\bf (C3)} For every $s\in\mathbb{R}$ the set $\mathcal{R}_s(\omega)=\cup_{r<s}\omega_{r,s}(\mathbb{R})$ is dense in $\mathbb{R}.$

\item {\bf (C4)} For all $s\leq t$ and $x\in\mathbb{R}$ the one-sided continuity 
$$
\omega_{s,t}(x)\in\{\omega_{s,t}(x-),\omega_{s,t}(x+)\}
$$
holds.

\item {\bf (C5)} For all $s\leq t$ and $x\not\in\mathcal{R}_s(\omega),$ 
$$
\omega_{s,t}(x)=\omega_{s,t}(x+).
$$

\end{itemize}
\end{definition}

\begin{remark} Each element $\omega\in \mathbb{F}$ is indeed a flow of mappings of $\mathbb{R}:$ evolutionary property is postulated in {\bf (C1)} while the condition $\omega_{s,s}(x)=x$ follows from the inclusion $\omega_{s,\cdot }(x)\in C_x([s,\infty)).$

\end{remark}

\begin{remark}  Condition {\bf (C1)} and continuity of trajectories $t\to \omega_{s,t}(x)$ imply that mappings $x\to \omega_{s,t}(x)$ are increasing. In particular, one-sided limits in {\bf (C4)} and {\bf (C5)} exist.

\end{remark}

\begin{remark}  Because of {\bf (C2)} all sets $\mathcal{R}_s(\omega)$ are countable.

\end{remark}

\begin{remark} Definition \ref{def_flow} is similar to \cite[Def. ]{Riabov}. Below we will show that {\bf (C1)-(C5)} actually imply conditions  from \cite[Def. 2.1]{Riabov}, so that $\mathbb{F}$  is a subset of the space $\mathbb{F}$ from \cite{Riabov}. This allows to  transfer results on measurability from \cite[L. 2.1]{Riabov}.

\end{remark}

\begin{remark} The space $\mathbb{F}$ is non-empty. We will show this in sections 4 and 6 by constructing a modification of the Arratia flow with drift as an $\mathbb{F}-$values random element. It is an interesting problem to give a direct example of  a flow $\omega\in\mathbb{F}.$ 

\end{remark}

In the next lemma we collect properties of a generic flow $\omega\in\mathbb{F}$ needed to equip $\mathbb{F}$ with nice measurability structure. 

\begin{lemma}
\label{lem1}

Consider arbitrary flow $\omega\in\mathbb{F}$ and real numbers $s,t,x$ with $s<t.$ Then

\begin{enumerate}
\item there exists $h>0$ such that either 
$$
\forall y\in[x-h,x] \  \omega_{s,t}(y)=\omega_{s,t}(x)
$$
or 
$$
\forall y\in[x,x+h] \  \omega_{s,t}(y)=\omega_{s,t}(x);
$$

\item there exists $r\in(s,t)$ and $y\in\mathbb{R}\setminus \mathcal{R}_r(\omega)$ such that  $\omega_{s,t}(x)=\omega_{r,t}(y).$
\end{enumerate}
\end{lemma}

\begin{proof}

\begin{enumerate}
\item Assume that $\omega_{s,t}(x)=\omega_{s,t}(x+).$ Using {\bf (C2)} we can find $\varepsilon>0$ such that 
$$
(\omega_{s,t}(x),\omega_{s,t}(x)+\varepsilon)\cap \omega_{s,t}(\mathbb{R})=\emptyset.
$$
Let $h>0$ be such that $\omega_{s,t}(y)<\omega_{s,t}(x)+\varepsilon$ for all $y\in[x,x+h].$ Necessarily we have $\omega_{s,t}(y)=\omega_{s,t}(x)$ for all $y\in[x,x+h].$ Similarly, in the case $\omega_{s,t}(x)=\omega_{s,t}(x-)$ there exists $h>0 $ such that $\omega_{s,t}(y)=\omega_{s,t}(x)$ for all $y\in[x-h,x].$ In the view of {\bf (C4)} these two cases are the only possible.

\item Assume that $\omega_{s,t}(x)=\omega_{s,t}(x+).$ There exists $z>x$  such that $\omega_{s,t}(z)=\omega_{s,t}(x).$ Using continuity of trajectories we can find $r\in(s,t)$ such that $\omega_{s,r}(z)>\omega_{s,r}(x).$ The range $\mathcal{R}_r(\omega)$ is countable, so there exists $y\in(\omega_{s,r}(x),\omega_{s,r}(z))\setminus \mathcal{R}_r(\omega).$ By monotonicity and evolutionary property {\bf (C1)}, $\omega_{r,t}(y)=\omega_{s,t}(x).$
\end{enumerate}

\end{proof}

We equip $\mathbb{F}$ with a cylindrical $\sigma-$field, i.e. $\mathcal{A}$ is the smallest $\sigma-$field that makes all mappings 
$$
\omega\to \omega_{s,t}(x)
$$
$\mathcal{A}/\mathcal{B}(\mathbb{R})$-measurable. Lemma \ref{lem1} implies that the space $\mathbb{F}$ is a subset of the space of flows from \cite[Def. 2.1]{Riabov}. The following result then follows from \cite[L. 2.1]{Riabov}

\begin{lemma}
\label{lem2} \cite[Lemma 2.1]{Riabov} Let $\mathcal{H}=\{(s,t)\in\mathbb{R}^2:s\leq t\}.$ The mapping 
$$
\mathcal{H}\times \mathbb{F}\times \mathbb{R}\ni (s,t,\omega,x)\to \omega_{s,t}(x)\in\mathbb{R}
$$
is $\mathcal{B}(\mathcal{H})\otimes \mathcal{A}\otimes \mathcal{B}(\mathbb{R})/\mathcal{B}(\mathbb{R})-$measurable
\end{lemma}

\begin{corollary}
\label{cor_lem2_1} Let $\theta_h:\mathbb{F}\to\mathbb{F}$ be a shift,
$$
(\theta_h\omega)_{s,t}(x)=\omega_{s+h,t+h}(x).
$$
Then the mapping 
$$
\mathbb{R}\times \mathbb{F}\ni (h,\omega)\to \theta_h\omega\in\mathbb{F}
$$
is $\mathcal{B}(\mathbb{R})\otimes \mathcal{A}/\mathcal{A}-$measurable. In other words, $(\theta_h)_{h\in\mathbb{R}}$ is a measurable group of transformations of $\mathbb{F}.$
\end{corollary}

\begin{corollary}
\label{cor_lem2_2} The mapping $\varphi:\mathbb{R}_+\times \mathbb{F}\times \mathbb{R}\to\mathbb{R}$,
$$
\varphi(t,\omega,x)=\omega_{0,t}(x),
$$
is a measurable perfect cocycle over $\theta.$

\end{corollary}

Perfect cocycle $\varphi$ naturally defines a flow of mappings of $\mathbb{R}$ by 
$$
\psi_{s,t}(\omega,x)=\varphi(t-s,\theta_s\omega,x).
$$
As it is mentioned in section 2, the perfect cocycle property implies the evolutionary property {\bf (SF2)} (definition \ref{def15_2}) of $\psi$ without exceptions in $\omega.$ In our construction, the flow reduces to 
$$
\psi_{s,t}(\omega,x)=\omega_{s,t}(x)
$$ 
and the evolutionary property holds without expections by the property {\bf (C1)} of the definition of the space $\mathbb{F}.$ Now we proceed with the construction of the dual flow. Advantage of the presented construction is that the dual flow is  constructed $\omega-$wise and for every $\omega$ it is indeed a flow of mappings of $\mathbb{R}.$ Moreover, the dual flow is generated by a backward perfect cocycle over $\theta.$  As  discussed in the introduction, there are two natural candidates for the dual flow:

\begin{itemize}
\item the family of right-continuous generalized  inverses
\begin{equation}
\label{eq07_1}
v^+_{t,s}(\omega,y)=\inf\{x\in\mathbb{R}:\omega_{s,t}(x)>y\};
\end{equation}

\item the family of left-continuous generalized  inverses
\begin{equation}
\label{eq07_2}
v^-_{t,s}(\omega,y)=\inf\{x\in\mathbb{R}:\omega_{s,t}(x)\geq y\}.
\end{equation}

\end{itemize}

Neither of them is a flow of mappings as the evolutionary property  {\bf (C1)} fails (see \cite{Arratia2} for examples).
Below we show that a proper choice between $v^+$ and $v^-$ gives rise to a dual flow. At first we need few properties of generalized inverses. 

\begin{lemma}
\label{lem 3} Consider a flow $\omega\in\mathbb{F}.$ Then

\begin{enumerate}
\item generalized inverses $v^\pm _{t,s}(\omega,y)$ are well-defined and finite for all $t\geq s$ and $y\in\mathbb{R};$

\item for each starting point $(t,y)\in\mathbb{R}^2$ mappings $s\to v^\pm_{t,s}(\omega,y)$ are continuous on $(-\infty,t]$ with   $v^\pm_{t,t}(\omega,y)=y;$

\item a backward flow of mappings $f=\{f_{t,s}:-\infty<s\leq t<\infty\}$ is dual to the flow $\omega$ if and only if 
$$
v^-_{t,s}(\omega,y)\leq f_{t,s}(y)\leq v^+_{t,s}(\omega,y)
$$
for all $t\geq s$ and $y\in\mathbb{R}.$

\end{enumerate}

\end{lemma}

\begin{proof} In the proof we omit the dependence of $v^\pm$ on $\omega.$ 

\begin{enumerate}
\item By definition, $v^\pm_{t,t}(y)=y.$ Let $t>s$ and $y\in\mathbb{R}.$ By condition {\bf (C2)} there are $x_1,x_2\in\mathbb{R}$ such that $\omega_{s,t}(x_1)<y<\omega_{s,t}(x_2).$ Monotonicity of $\omega_{s,t}$ implies  
$$
[x_2,\infty)\subset \{x:\omega_{s,t}(x)>y\}\subset \{x:\omega_{s,t}(x)\geq y\}\subset (x_1,\infty).
$$
Infima in \eqref{eq07_1} and \eqref{eq07_2} are finite:
$$
x_1\leq v^-_{t,s}(y)\leq v^+_{t,s}(y)\leq x_2.
$$

\item We prove continuity of $v^+_{s,t}(y)$ at a point $s<t.$ Proofs for $v^-$ and $s=t$ are similar. Let $\varepsilon>0.$ Using {\bf (C3)} we can find  $r<s$ and $x_1<x_2$  such that 
$$
v^+_{t,s}(y)-\varepsilon<\omega_{r,s}(x_1)<v^+_{t,s}(y)<\omega_{r,s}(x_2)<v^+_{t,s}(y)+\varepsilon.
$$
By continuity of trajectories $t\to\omega_{r,t}(x)$ there exists $\delta\in(0,\min(s-r,t-s))$ such that for all $u\in[s-\delta,s+\delta]$
\begin{equation}
\label{eq07_3}
v^+_{t,s}(y)-\varepsilon<\omega_{r,u}(x_1)<v^+_{t,s}(y)<\omega_{r,u}(x_2)<v^+_{t,s}(y)+\varepsilon.
\end{equation}
Then from the definition of $v^+$ and the evolutionary property of $\omega,$
$$
\omega_{r,s}(x_1)<v^+_{t,s}(y)\Rightarrow \omega_{s,t}(\omega_{r,s}(x_1))=\omega_{r,t}(x_1)\leq y
$$
$$
\Rightarrow \omega_{u,t}(\omega_{r,u}(x_1))\leq y\Rightarrow \omega_{r,u}(x_1)\leq v^+_{t,u}(y).
$$
Similarly, $\omega_{r,u}(x_2)\geq v^+_{t,u}(y)$ and 
$$
\omega_{r,u}(x_1)\leq v^+_{t,u}(y)\leq \omega_{r,u}(x_2).
$$
From inequalities \eqref{eq07_3} we deduce that for all $u\in[s-\delta,s+\delta],$
$$
|v^+_{t,u}(y)-v^+_{t,s}(y)|\leq \varepsilon.
$$

Statement 3 is merely a reformulation of the definition \ref{def15_1} of duality.
\end{enumerate}

\end{proof}

The following definition is taken from \cite{Arratia2}.

\begin{definition}
\label{def_reg_point} A point $(t,y)\in\mathbb{R}^2$ is said to be left  regular for a flow $\omega\in\mathbb{F},$ if for all $u\geq t$ $\omega_{t,u}(y)=\omega_{t,u}(y-).$ Otherwise a point $(t,y)\in\mathbb{R}^2$ is  said to be left irregular for $\omega.$
\end{definition}

\begin{remark} A point $(t,y)$ is left regular for $\omega$ if and only if there exist two rational sequences $(u_n)_{n\geq 1}$ and $(y_n)_{n\geq 1}$ such that $u_n>t,$ $y_n<y,$ $\lim_{n\to\infty}u_n=t$ and $\omega_{t,u_n}(y_n)=\omega_{t,u_n}(y)$ (see Lemma \ref{lem1}).  In view of the property {\bf (C4)}, if a point $(t,y)$ is left irregular   for $\omega,$ then there exist two rational sequences $(u_n)_{n\geq 1}$ and $(y_n)_{n\geq 1}$ such that $u_n>t,$ $y_n>y,$ $\lim_{n\to\infty}u_n=t$ and $\omega_{t,u_n}(y_n)=\omega_{t,u_n}(y).$

\end{remark}

In the next theorem we construct a backward flow of mappings $\tilde{\psi}(\omega)=\{\tilde{\psi}_{t,s}(\omega,\cdot):-\infty<s\leq t<\infty$ such that for every $\omega$ $\tilde{\psi}$ is dual to $\omega.$ The result extends  \cite[Section 6]{Arratia2}.

\begin{theorem}
\label{thm2} For all $s\leq t,$ $y\in\mathbb{R}$ and $\omega\in\mathbb{F}$ set 
$$
\tilde{\psi}_{t,s}(\omega,y)=\begin{cases}
v^+_{t,s}(\omega,y), \mbox{ if the point } (t,y) \mbox{ is left regular for } \omega \\
v^-_{t,s}(\omega,y), \mbox{ if the point } (t,y) \mbox{ is left  irregular for } \omega 
\end{cases}
$$
Then 
\begin{enumerate}
\item the mapping $(s,t,\omega,y)\to \tilde{\psi}_{t,s}(\omega,y)$ is jointly measurable;

\item for every $\omega\in\mathbb{F},$  $\tilde{\psi}(\omega)=\{\tilde{\psi}_{t,s}(\omega,\cdot):-\infty<s\leq t<\infty$ is a backward flow dual to $\omega;$

\item for all $\omega\in\mathbb{F},$ $t\geq s,$ $y\in \mathbb{R}$ and $h\in\mathbb{R},$
\begin{equation}
\label{eq07_4}
\tilde{\psi}_{t,s}(\theta_h\omega,y)=\tilde{\psi_{t+h,s+h}}(\omega,y).
\end{equation}
\end{enumerate}

\end{theorem}

\begin{proof} Note that \eqref{eq07_4} is an immediate consequence of the definitions. Since $(h,\omega)\to\theta_h\omega$ is jointly measurable, the joint measurability of $\tilde{\psi}_{t,s}(\omega,y)$ follows from the joint measurability of 
$$
(s,\omega,y)\to\tilde{\psi}_{0,s}(y).
$$
Let 
$$
A=\{(\omega,y)\in\mathbb{F}\times\mathbb{R}:(0,y) \mbox{ is left regular for } \omega\}.
$$
Measurability of $A$ follows from the representation
$$
A=\bigcap_{q\in\mathbb{Q},q>0}\bigcup_{x\in\mathbb{Q}}\bigg(\mathbb{F}\times (x,\infty)\cap \{(\omega,y)|\omega_{0,q}(y)=\omega_{0,q}(x)\}\bigg).
$$
Since $\tilde{\psi}_{0,s}(\omega,y)=v^\pm_{0,s}(\omega,y)$ depending on whether $(\omega,y)\in A$ or not, it is enough to check joint measurability of 
$$
(s,\omega,y)\to v^\pm_{0,s}(y).
$$
The latter follows from equivalences 
$$
v^+_{0,s}(\omega,y)<c\Leftrightarrow \exists \mbox{ rational } q<c: \omega_{s,0}(q)>y;
$$
$$
v^-_{0,s}(\omega,y)<c\Leftrightarrow \exists \mbox{ rational } q<c: \omega_{s,0}(q)\geq y.
$$
The property 1)  is proved.

Now we check that for any $\omega\in\mathbb{F}$ the family of mappings $\{\tilde{\psi}_{t,s}(\omega,\cdot ): -\infty<s\leq t<\infty\}$ is a backward flow of mappings of $\mathbb{R},$ i.e. that the evolutionary  property  holds. Since $\tilde{\psi}_{t,t}(\omega,y)=y,$ it is enough to consider the case $r<s<t.$ Denote $\tilde{y}=\tilde{\psi}_{t,s}(\omega,y),$ $x=\tilde{\psi}_{t,r}(\omega,y),$ $\tilde{x}=\tilde{\psi}_{s,r}(\omega,\tilde{y}).$

Assume $x<\tilde{x}$ and let $z\in(x,\tilde{x}).$ From inequalities 
$$
z>x=\tilde{\psi}_{t,r}(\omega,y)\geq v^-_{t,r}(\omega,y)
$$
and
$$
z<\tilde{x}=\tilde{\psi}_{s,r}(\omega,\tilde{y})\leq v^+_{s,r}(\omega,\tilde{y})
$$ 
it follows that $\omega_{r,t}(z)\geq y,$ $\omega_{r,s}(z)\leq \tilde{y}.$ Assume that $\omega_{r,s}(z)<\tilde{y}.$ Since $\tilde{y}\leq v^+_{t,s}(\omega,y),$ we have $\omega_{s,t}(\omega_{r,s}(z))=\omega_{r,t}(z)\leq y.$  Hence, $\omega_{r,t}(z)=y.$ Denote $\tilde{z}=\omega_{r,s}(z).$ We have obtained relations 
$$
\tilde{\psi}_{t,r}(\omega,y)<z, \ \omega_{r,t}(z)=y 
$$
$$
\tilde{z}<\tilde{\psi}_{t,s}(\omega,y), \ \omega_{s,t}(\tilde{z})=y 
$$
Then $\tilde{\psi}_{t,r}(\omega,y)<z\leq v^+_{r,t}(\omega,y)$ and $\tilde{\psi}_{t,r}(\omega,y)\ne v^+_{t,r}(\omega,y).$ The point $(t,y)$ is left irregular for $\omega.$ But also $v^-_{t,s}(\omega,y)\leq \tilde{z}<\tilde{\psi}_{t,s}(\omega,y)$ and $\tilde{\psi}_{t,s}(\omega,y)\ne v^-_{t,r}(\omega,y).$ The point $(t,y)$ is left regular for $\omega,$ which is impossible.

Obtained contradiction shows that  $\omega_{r,s}(z)=\tilde{y}=\tilde{\psi}_{t,s}(\omega,y),$ $z\geq v^-_{s,r}(\omega,\tilde{y}).$   From inequalities 
$$
\tilde{\psi}_{s,r}(\omega,\tilde{y})>z\geq v^-_{s,r}(\omega,\tilde{y})
$$
it follows that the point $(s,\tilde{y})$ is left regular for $\omega.$ there exists $\hat{y}<\tilde{y}$ such that $\omega_{s,t}(\hat{y})=\omega_{s,t}(\tilde{y}).$ Further, 
$$
\hat{y}<\tilde{y}=\tilde{\psi}_{t,s}(\omega,y)\leq v^+_{t,s}(\omega,y)
$$
and $\omega_{s,t}(\hat{y})\leq y.$ On the onther hand,
$$
\omega_{s,t}(\hat{y})=\omega_{s,t}(\tilde{y})=\omega_{r,t}(z)\geq y.
$$ 
It means that 
$$
\omega_{r,t}(z)=\omega_{s,t}(\tilde{y})=\omega_{s,t}(\hat{y})=y,
$$
and
$$
\tilde{y}=\tilde{\psi}_{t,s}(\omega,y)>\hat{y}\geq v^-_{t,s}(\omega,y).
$$
Since $\tilde{\psi}_{t,r}(\omega,y)<z\leq v^+_{t,r}(\omega,y),$ we deduce that again the point $(t,y)$ is both left regular and left irregular for $\omega.$ The case $x<\tilde{x}$ is impossible. 

Considerations in the case $x>\tilde{x}$ are similar.
 
\end{proof}

\begin{corollary}
\label{cor_thm1} The mapping 
$$
\tilde{\varphi}:\mathbb{R}_+\times \mathbb{F}\times\mathbb{R}\to\mathbb{R}, \ \tilde{\varphi}(t,\omega,x)=\tilde{\psi}_{t,0}(\omega,x)
$$
is a perfect backward cocycle over $\theta.$

\end{corollary}

\section{Coalescing stochastic flow as a random element in $\mathbb{F}$}

In this section starting from a sequence $\{P^{(n)}:n\geq 1\}$ of transition probabilities for the $n-$point motions we construct a probability measure on the space $(\mathbb{F},\mathcal{A})$ that makes $\psi_{s,t}(\omega,x)=\omega_{s,t}(x)$ a stochastic flow with finite-point motions defined by $\{P^{(n)}:n\geq 1\}.$ We recall assumptions on finite-point motions.

\begin{itemize}
\item {\bf (TP1)} Each $P^{(n)}=\{P^{(n)}_t:t\geq 0\}$ is a Feller trasition probability on $(\mathbb{R}^n,\mathcal{B}(\mathbb{R}^n)).$

\item {\bf (TP2)} Given $1\leq i_1<i_2<\ldots<i_k\leq n$ let $\pi_{i_1,\ldots,i_k}:\mathbb{R}^n\to\mathbb{R}^k$ be a projection, $\pi_{i_1,\ldots,i_k}(x)=(x_{i_1},\ldots,x_{i_k}).$ The for all $t\geq 0, $$x\in\mathbb{R}^n$ and $C\in\mathcal{B}(\mathbb{R}^k),$
$$
P^{(n)}_t(x,\pi^{-1}_{i_1,\ldots,i_k}C)=P^{(k)}_t(\pi_{i_1,\ldots,i_k}x,C).
$$

\item {\bf (TP3)} Let $\Delta=\{(y,y):y\in\mathbb{R}\}$  be a diagonal in $\mathbb{R}^2,$ then for all $t\geq 0$ and $x\in \Delta$
$$
P^{(2)}_t(x,\Delta)=1.
$$

\item {\bf (TP4)} For all $t>0,$ $x,y\in\mathbb{R}$
$$
P^{(1)}_t(x,\{y\})=0.
$$

\item {\bf (TP5)} For all real $a<b$ and $\varepsilon>0$
$$
\lim_{t\to 0}t^{-1}\sup_{x\in[a,b]}P^{(1)}_t(x,(x-\varepsilon,x+\varepsilon)^c)=0
$$

\item {\bf (TP6)} Given reals $a<b$ and $t>0$ there exists an increasing continuous function $m_{a,b,t}:\mathbb{R}\to\mathbb{R}$ such that for all $x_1,x_2$ with  $a\leq x_1<x_2\leq b,$
$$
\mathbb{P}^{(2)}_{(x_1,x_2)}(\forall s\in [0,t] \ a\leq X^{(2)}_1(s)<X^{(2)}_2(s)\leq b)\leq m_{a,b,t}(x_2)-m_{a,b,t}(x_1).
$$

\item {\bf (TP7)} Given reals $a<b$ and $t>0$ there exists a positive function $f_{a,b,t}:(0,\infty)\to(0,\infty)$ such that for all $x_1,x_2,x_3$ with  $a\leq x_1<x_2<x_3\leq b,$
$$
\mathbb{P}^{(3)}_{(x_1,x_2,x_3)}(\forall s\in [0,t] \ a\leq X^{(3)}_1(s)<X^{(3)}_2(s)<X^{(3)}_3(s)\leq b)\leq f_{a,b,t}(x_3-x_1).$$

\end{itemize}

We will use the function  
$$
w_{a,b}(\varepsilon,\delta)=\inf\{t>0: \sup_{x\in[a,b]}P^{(1)}_t(x,(x-\varepsilon,x+\varepsilon)^c)\geq \delta t\}
$$
defined in section 2 \eqref{eq08_1}.

\begin{theorem}
\label{thm3} Let $\{P^{(n)}:n\geq 1\}$ be a compatible sequence of coalescing Feller transition probabilities satisfying conditions {\bf (TP1)-(TP7)}.
 Assume that for any reals $a<b$ and $t>0$
 $$
 \liminf_{\varepsilon,\delta\to 0}\frac{f_{a,b,t}(8\varepsilon)}{w_{a,b}(\varepsilon,\delta)}=0.
 $$
 Then there exists an $\mathbb{F}-$valued random element $\psi=\{\psi_{s,t}:-\infty<s\leq t<\infty\},$ which is a stochastic flow with finite-point motions determined by $\{P^{(n)}:n\geq 1\}.$
\end{theorem}

\begin{proof}
 From \cite{Riabov} it follows that there exists a stochastic flow $\psi=\{\psi_{s,t}:-\infty<s\leq t<\infty\}$ with finite-point motions determined by $\{P^{(n)}:n\geq 1\}$  and such that all realizations $\{\psi_{s,t}(\omega,\cdot):-\infty<s\leq t<\infty\}$ satisfy 
 
(i) properties  {\bf (C1),(C3),(C5)} of the definition \ref{def_flow} \cite[Th. 1.1]{Riabov};
 
(ii) for all $s<t$ and $a<b$ images $\psi_{s,t}([a,b])$ are finite \cite[L. 3.2, property {\bf SP4}]{Riabov};
 
(iii) for any $s<t$ and $x\in\mathbb{R}$ there exist $r<t$ and $y\not\in\mathcal{R}_r(\psi)$ such that $\psi_{s,t}(x)=\psi_{r,t}(y)$ \cite[Th. 1.1]{Riabov};

(iv) given rationals $v_1<v_2<v_3$ and $p_1<p_2$ for all large enough  integers $N\geq 1$ and all $j=0,\ldots,N-1$ one has 
$$
v_1<\psi_{q_j,t}(v_2)<v_3, \ q_j\leq t\leq q_{j+1},
$$
where $q_j=p_1+\frac{j(p_2-p_1)}{N},$ $0\leq j\leq N$ \cite[L. 3.2, proof of {\bf SP3}]{Riabov}.

It remains to check that outside a set of probability zero properties {\bf (C2), (C4)} are satisfied. {\bf (C2)} is satisfed on the event 
$$
\bigcap^\infty_{n=1}\bigg(\{\sup_{k\in\mathbb{Z}}\psi_{-n,n}(k)=\infty\}\cap \{\inf_{k\in\mathbb{Z}}\psi_{-n,n}(k)=-\infty\}\bigg).
$$
The probability of the latter event is $1$ since Feller property {\bf (TP1)} and continuity of trajectories {\bf (TP4)} imply
$$
\lim_{x\to\infty}P^{(1)}_t(x,[c,\infty))=1 \mbox{ and }\lim_{x\to-\infty}P^{(1)}_t(x,(-\infty,c])=1
$$ 
for all $t\geq 0$ and $c\in\mathbb{R}$ (the proof is postponed to the appendix).

Condition {\bf (C4)} is satisfies at all points $(s,y)$ with $y\not\in\mathcal{R}_s(\psi),$ because the property {\bf (C5)} holds. So, it is enough to check {\bf (C4)} at all points $(t,x)$ where $x=\psi_{s,t}(y),$ $s<t.$ By the property (iii) above we can assume that $y\not\in\mathcal{R}_s(\psi).$ Moreover, using $\psi_{s,t}(y)=\psi_{s,t}(y+)$ and (ii) we can assume $(s,y)\in\mathbb{Q}^2$ (see the proof of the lemma \ref{lem1}). Thus, it is enough to check that for all $(s,y)\in\mathbb{Q}^2,$ $M\geq 1$ and $\eta\in\mathbb{Q},$ $\eta>0,$
outside an event of probability zero for all $t\in[s,s+M-\eta]$
$$
\psi_{s,t+\eta}(y)\in\{\psi_{t,t+\eta}(\psi_{s,t}(y)-),\psi_{t,t+\eta}(\psi_{s,t}(y)+)\}.
$$
To prove this assertion we adopt an approach of \cite{TW}.  Introduce a set 
$$
A_m=\{\forall u\in [s,s+M+\eta] \forall x\in (\psi_{s,u}(y)-1,\psi_{s,u}(y)+1) \ \ \sup_{t\in [u,s+M+\eta]}|\psi_{u,t}(x)|\leq m\}
$$
(it is measurable since one can restrict $u,x$ to take rational values without changing the event). 

We observe that $A_m\uparrow \Omega,$ $m\to\infty.$ Indeed, for fixed $\omega$ we can find rational numbers $u,v$ such that for all $t\in [s,s+M+\eta]$
$$
u<\psi_{s,t}(y)-1<\psi_{s,t}(y)+1<v.
$$ 
Using the property (iv) above we can find integer $N$ such that for all $j=0,\ldots,N-1$  and $t\in[q_j,q_{j+1}]$
$$
u-2<\psi_{q_j,t}(u-1)<u,  \  v<\psi_{q_j,t}(v+1)<v+2,
$$
where $q_j=s+\frac{j(M+\eta)}{N}.$ By continuity of trajectories there exists  $m\geq 1$ such that for all $j=0,\ldots,N-1$ and $t\in [q_j,M+\eta]$
$$
-m\leq \psi_{q_j,t}(u-1)\leq \psi_{q_j,t}(v+1)\leq m.
$$
By evolutionary property and construction of the points $\{q_0,\ldots,q_N\}$ we get that the event $A_m$ happens.

Further, let $\varepsilon_n,\delta_n\to 0$ be such that 
$$
\frac{f_{-m,m,\eta/2}(8\varepsilon_n)}{w_{-m,m}(\varepsilon_n,\delta_n)}\to 0, \ n\to\infty.
$$
We check that $w_{-m,m}(\varepsilon_n,\delta_n)\to 0.$ Assume it is not the case. Passing to subsequences we may assume that 
$$
\inf_{n\geq 1}w_{-m,m}(\varepsilon_n,\delta_n)>0.
$$
Using the definition of the function $w_{-m,m}$ we find $t>0$ and a sequence $x_n\in[-m,m]$ such that 
$$
P^{(1)}_t(x_n,(x_n-\varepsilon_n,x_n+\varepsilon_n)^c)<\delta_n t.
$$
In particular, 
\begin{equation}
\label{eq19_1}
\lim_{n\to\infty}P^{(1)}_t(x_n,(x_n-\varepsilon_n,x_n+\varepsilon_n)^c)=0.
\end{equation}
Extracting another subsequence we may assume that $x_n\to x\in[-m,m].$ The Feller property implies the weak convergence  \cite[L. 19.3]{Kallenberg}
$$
P^{(1)}_t(x_n,\cdot)\to P^{(1)}_t(x,\cdot), \ n\to\infty.
$$
Now  \eqref{eq19_1} implies that $P^{(1)}_t(x,\{x\})=1$ which contradicts {\bf (TP4)}. So,
$$
\lim_{n\to\infty}w_{-m,m}(\varepsilon_n,\delta_n) =0.
$$

Let $K=M(1+[w_{-m,m}(\varepsilon_n,\delta_n)^{-1}]).$ Consider points $q_j=s+\frac{j M}{K},$ $\xi_j=\psi_{s,q_j}(y),$ $0\leq j\leq K$. Introduce two events
$$
B_{m,n}=\{\forall j\in\{0,\ldots,K-1\}\forall l\in\{-1,0,1\} \ \ \  |\psi_{q_j,q_{j+1}}(\xi_j+2\varepsilon_n l)-(\xi_j+2\varepsilon_n l)|<\varepsilon_n\}
$$
$$
C_{m,n}=\{\exists j\in\{0,\ldots,K-1\} \ \  \psi_{q_j,q_j+\eta/2}(\xi_j-4\varepsilon_n)<\psi_{q_j,q_j+\eta/2}(\xi_j)<\psi_{q_j,q_j+\eta/2}(\xi_j+4\varepsilon_n)\}
$$
For large enough $n,$ we have $\varepsilon_n<\frac{1}{4},$ and $w_{-m,m}(\varepsilon_n,\delta_n)<1.$ Then
$$
\mathbb{P}(A_m\cap C_{m,n})\leq Kf_{-m,m,\frac{\eta}{2}}(8\varepsilon_n)\leq 2M\frac{f_{-m,m,\frac{\eta}{2}}(8\varepsilon_n)}{w_{-m,m}(\varepsilon_n,\delta_n)}
$$
and
$$
\mathbb{P}(A_m\setminus  B_{m,n})\leq 3K\sup_{|x|\leq m}P^{(1)}_{\frac{M}{K}}(x,(x-\varepsilon_n,x+\varepsilon_n)^c)\leq 3M\delta_n,
$$
where the last inequality follows from $\frac{M}{K}<w_{-m,m}(\varepsilon_n,\delta_n)$. Then
$$
\mathbb{P}(A_m\cap B_{m,n}\setminus C_{m,n})\geq \mathbb{P}(A_m)-2M\frac{f_{-m,m,\frac{\eta}{2}}(8\varepsilon_n)}{w_{-m,m}(\varepsilon_n,\delta_n)}-3M\delta_n
$$
and 
$$
\mathbb{P}\bigg(\bigcup_{m\geq 1}\bigg(A_m\cap \limsup_{n\to\infty}\bigg(B_{m,n}\setminus C_{m,n}\bigg)\bigg)\bigg)=1
$$
Assume that the latter event happens, but for some $t\in [s,s+M-\eta]$ we have 
\begin{equation}
\label{eq13_1}
\psi_{t,t+\eta}(\psi_{s,t}(y)-)<\psi_{s,t+\eta}(y)<\psi_{t,t+\eta}(\psi_{s,t}(y)+).
\end{equation}
Let $m$ and $n$ be such that the event $A_m\cap (B_{m,n}\setminus C_{m,n})$ happens and $\frac{M}{K}<\frac{\eta}{2}.$ There is $j\in[0,K-2]$ such that $q_j\leq t<q_{j+1}.$ By the definition of the event $B_{m,n}$
$$
\xi_j-3\varepsilon_n<\psi_{q_j,q_{j+1}}(\xi_j-2\varepsilon_n)<\xi_j-\varepsilon_n<\psi_{q_j,q_{j+1}}(\xi_j)<
$$
$$
<\xi_j+\varepsilon_n<\psi_{q_j,q_{j+1}}(\xi_j+2\varepsilon_n)<\xi_j+3\varepsilon_n.
$$
It follows that 
$$
\psi_{q_j,t}(\xi_j-2\varepsilon_n)<\psi_{q_j,t}(\xi_j)=\psi_{s,t}(y)<\psi_{q_j,t}(\xi_j+2\varepsilon_n).
$$
From \eqref{eq13_1} we deduce 
$$
\psi_{q_j,t+\eta}(\xi_j-2\varepsilon_n)<\psi_{q_j,t+\eta}(\xi_j)=\psi_{s,t+\eta}(y)<\psi_{q_j,t+\eta}(\xi_j+2\varepsilon_n).
$$
Further, 
$$
\xi_{j+1}-4\varepsilon_n<\xi_{j}-3\varepsilon_n<\psi_{q_j,q_{j+1}}(\xi_j-2\varepsilon_n),
$$
$$
\xi_{j+1}+4\varepsilon_n>\xi_{j}+3\varepsilon_n<\psi_{q_j,q_{j+1}}(\xi_j+2\varepsilon_n),
$$
and
$$
\psi_{q_{j+1},t+\eta}(\xi_{j+1}-4\varepsilon_n)\leq \psi_{q_{j+1},t+\eta}(\psi_{q_j,q_{j+1}}(\xi_j-2\varepsilon_n))=
$$
$$
=\psi_{q_j,t+\eta}(\xi_j-2\varepsilon_n)<\psi_{s,t+\eta}(y)=\psi_{q_{j+1},t+\eta}(\xi_{j+1})<
$$
$$
<\psi_{q_{j+1},t+\eta}(\psi_{q_j,q_{j+1}}(\xi_j+2\varepsilon_n))\leq \psi_{q_{j+1},t+\eta}(\xi_{j+1}+4\varepsilon_n)
$$
But $t+\eta>q_{j+1}+\frac{\eta}{2}$ which means that the event $C_{m,n}$ happens. This contradiction shows that outside an event of probability  zero the flow $\{\psi_{s,t}:-\infty<s\leq t<\infty\}$ satisfies all conditions of the definition \ref{def_flow} and can be considered as an $\mathbb{F}-$valued random element.

\end{proof}

\begin{corollary}
\label{cor_thm2} Under assumptions of the theorem \ref{thm3} there is a unique probability measure $\mu$ on the space $(\mathbb{F},\mathcal{A})$ such that $(\mathbb{F},\mathcal{A},\mu,(\theta_h)_{h\in\mathbb{R}})$ is a metric dynamical system, $\varphi$ is  a forward perfect cocycle over $\theta$ that generates a stochastic flow on $\mathbb{R}$ with finite point motions determined by transition probabilities $\{P^{(n)}:n\geq 1\}.$

\end{corollary}

\section{Distribution of the dual flow}

As shown in the section 2, the metric dynamical system $(\mathbb{F},\mathcal{A},\mu,(\theta_h)_{h\in\mathbb{R}})$ carries a backward perfect cocylce $\tilde{\varphi},$ such that for every $\omega\in\mathbb{F}$ the backward flow $\{\tilde{\psi}_{t,s}(\omega,\cdot):-\infty<s\leq t<\infty\}$ is dual to the flow $\{\psi_{t,s}(\omega,\cdot):-\infty<s\leq t<\infty\}.$ In this section we prove that $\tilde{\psi}$ is a backward stochastic flow and describe transition semigroups for its finite point motions -- they are dual to transition semigroups of the flow $\psi$ in the sense of \cite[Section 2, \S 3]{Liggett}.

\begin{theorem}
\label{thm4} Under assumptions of the theorem \ref{thm2} there exists a unique compatible  sequence $\{\tilde{P}^{(n)}:n\geq 1\}$ of coalescing Feller transition probabilities on $\mathbb{R}$ such that for all $n\geq 1$ and all $x,y\in\mathbb{R}^n$ with $x_1<y_1<x_2<y_2<\ldots<x_n<y_n,$
\begin{equation}
\label{eq21_2}
\begin{aligned}
\tilde{P}^{(n)}(y,(x_1,x_2)\times(x_2,x_3)\times\ldots\times (x_n,\infty))= \\
=P^{(n)}(x,(-\infty,y_1)\times(y_1,y_2)\times\ldots\times (y_{n-1},y_n)).
\end{aligned}
\end{equation}
Further, $\tilde{\psi}$ is a backward stochastic flow on $\mathbb{R}$ with finite point motions determined by transition probabilities $\{\tilde{P}^{(n)}:n\geq 1\}.$

\end{theorem}

\begin{proof} We observe that the $\sigma-$field $\sigma(\{\tilde{\psi}_{v,u}:s\leq u\leq v\})$ is contained in $\mathcal{F}^\psi_{s,\infty}.$ At first we prove the following relation.

Let  $s\leq t,$ $n\geq 1,$ $\xi_1,\ldots,\xi_n$ are $\mathcal{F}^\psi_{s,\infty}-$measurable random variables  and  $x_1<x_2<\ldots<x_n.$ Then  
\begin{equation}
\label{eq20_1}
\begin{aligned}
\mu(x_1<\tilde{\psi}_{t,s}(\xi_1)<x_2<\tilde{\psi}_{t,s}(\xi_2)<x_3<\ldots<x_n<\tilde{\psi}_{t,s}(\xi_n)|\mathcal{F}^{\psi}_{s,\infty})=\\
=P^{(n)}_{t-s}((x_1,\ldots,x_n),(-\infty,\xi_1)\times (\xi_1,\xi_2)\times\ldots\times (\xi_{n-1},\xi_n)).
\end{aligned}
\end{equation}
Indeed, inequalities $v^-_{t,s}(y)\leq \tilde{\psi}_{t,s}(y)\leq v^+_{t,s}(y)$ imply inclusions
$$
\{\psi_{s,t}(x_1)<\xi_1<\psi_{s,t}(x_2)<\ldots<\xi_{n-1}<\psi_{s,t}(x_n)<\xi_n\}\subset 
$$
$$
\subset \{x_1<\tilde{\psi}_{t,s}(\xi_1)<x_2<\tilde{\psi}_{t,s}(\xi_2)<x_3<\ldots<x_n<\tilde{\psi}_{t,s}(\xi_n)\}\subset
$$
$$
\subset \{\psi_{s,t}(x_1)\leq \xi_1\leq \psi_{s,t}(x_2)\leq \ldots\leq \xi_{n-1}\leq \psi_{s,t}(x_n)\leq \xi_n\}.
$$
The  relation \eqref{eq20_1} then follows from the definition of $\psi$ and the property {\bf (TP4)}.   

Our assumption on the meeting of two trajectories implies that for all $t>0,$ $x,y\in\mathbb{R}$ 
\begin{equation}
\label{eq21_1}
\mu(\tilde{\psi}_{t,0}(y)=x)=0.
\end{equation}
Indeed, by construction of the dual flow and the property {\bf (TP4)}
$$
\mu(\tilde{\psi}_{t,0}(y)=x)\leq \mu(v^-_{t,0}(y)<x+\varepsilon, v^+_{t,0}(y)>x-\varepsilon)\leq 
$$
$$
\leq \mu(\psi_{0,t}(x+\varepsilon)\geq y \geq \psi_{0,t}(x-\varepsilon))\leq 
$$
$$
\leq \mathbb{P}^{(2)}_{x-\varepsilon,x+\varepsilon}(\forall s\in[0,t] \ \  X^{(2)}_1(s)<X^{(2)}_2(s)).
$$
By the property {\bf (TP6)} the latter probability tends to zero as $\varepsilon\to 0.$ 

For every $n\geq 1$ and $t\geq 0$ we introduce a family $\{\tilde{P}^{(n)}_t(y,\cdot):y\in\mathbb{R}^n\}$ of probability measures on $(\mathbb{R}^n,\mathcal{B}(\mathbb{R}^n))$ by 
$$
\tilde{P}^{(n)}_t(y,B)=\mu((\tilde{\psi}_{t,0}(y_1),\ldots,\tilde{\psi}_{t,0}(y_n))\in B), \ B\in\mathcal{B}(\mathbb{R}^n).
$$
Then conditions {\bf (TP2), (TP3)} are satisfied. 

Inductively on $n$ we will check that each $\{\tilde{P}^{(n)}_t:t\geq 0\}$ is a Feller transition probability on $\mathbb{R}^n,$ and that $\tilde{P}^{(n)}_t(y,\cdot)$ is the distribution of $(\tilde{\psi}_{t+h,h}(y_1),\ldots,\tilde{\psi}_{t+h,h}(y_n))$ for all $t,h\in\mathbb{R}$ and $y\in\mathbb{R}^n.$   Consider the case $n=1.$ Using \eqref{eq20_1} with $n=1$ and non-random $\xi=y$ we get 
$$
\mu(\tilde{\psi}_{t,s}(y)>x)=P^{(1)}_{t-s}(x,(-\infty,y))=\mu(\tilde{\psi}_{t-s,0}(y)>x)=\tilde{P}^{(1)}_{t-s}(y,(x,\infty)).
$$
Consequently, for all $h\in\mathbb{R}$ the distribution of $\tilde{\psi}_{t+h,h}(y)$ is  $\tilde{P}^{(1)}_{t}(y,\cdot).$ Further,  applying \eqref{eq20_1} with $\xi=\psi_{t+s,s}(y)$ we get
$$
\tilde{P}^{(1)}_{t+s}(y,(x,\infty))=\mu(\tilde{\psi}_{t+s,0}(y)>x)=\mathbb{E}_\mu \mu(\tilde{\psi}_{s,0}(\tilde{\psi}_{t+s,s}(y))>x|\mathcal{F}^\psi_{s,\infty})=
$$
$$
=\mathbb{E}_\mu P^{(1)}_{s}(x,(-\infty,\tilde{\psi}_{t+s,s}(y)))=
$$
$$
=\int_{\mathbb{R}}P^{(1)}_s(x,(-\infty,z))\tilde{P}^{(1)}_{t}(y,dz)=\int_{\mathbb{R}}\tilde{P}^{(1)}_s(z,(x,\infty))\tilde{P}^{(1)}_{t}(y,dz).
$$
This proves the Chapman-Kolmogorov equation for the family $\tilde{P}^{(1)}.$ In order to check Feller property, consider a continuously differentiable function with compact support $f:\mathbb{R}\to\mathbb{R},$ $\mbox{supp}(f)\subset [a,b].$ From the representation
$$
\int_{\mathbb{R}}f(x)\tilde{P}^{(1)}_t(y,dx)=\int^b_a f'(z)\tilde{P}^{(1)}_t(y,(z,\infty))dz=
$$
$$
=\int^b_a f'(z)P^{(1)}_t(z,(-\infty,y))dz,
$$
the property {\bf (TP4)} and the dominated convergence theorem, we deduce that the function 
$$
y\to \int_{\mathbb{R}}f(x)\tilde{P}^{(1)}_t(y,dx)
$$
belongs to $C_0(\mathbb{R}).$  By a standard density argument,  $\tilde{P}^{(1)}$ is a Feller transition probability on $\mathbb{R}.$ 

Assume that the result is proved for all $k\leq n-1.$ Let $y\in\mathbb{R}^n.$ By consistency property {\bf (TP2)} and coaelscing condition {\bf (TP3)} it is enough to consider the case $y_1<y_2<\ldots<y_n.$ We prove that the law of $(\tilde{\psi}_{t+h,h}(y_1),\ldots,\tilde{\psi}_{t+h,h}(y_n))$ is $\tilde{P}^{(n)}_t(y,\cdot)$ once we check that for any $x\in\mathbb{R}^n$
$$
\mu(\tilde{\psi}_{t+h,h}(y_1)>x_1,\ldots,\tilde{\psi}_{t+h,h}(y_n)>x_n)=\tilde{P}^{(n)}_t(y,\prod^n_{j=1}(x_j,\infty)).
$$
By monotonicity of trajectories and inductive assumption it is enough to consider the case $x_1<x_2<\ldots<x_n.$ Equation \eqref{eq21_1} and monotonicity of trajectories implies the representation
$$
\mu(\tilde{\psi}_{t+h,h}(y_1)>x_1,\ldots,\tilde{\psi}_{t+h,h}(y_n)>x_n)=\mu(\tilde{\psi}_{t+h,h}(y_1)>x_n)+
$$
$$
+\mu(x_{n-1}<\tilde{\psi}_{t+h,h}(y_1)<x_n<\tilde{\psi}_{t+h,h}(y_n))+\ldots+
$$
$$
+\mu(x_{1}<\tilde{\psi}_{t+h,h}(y_1)<x_2<\tilde{\psi}_{t+h,h}(y_2)<\ldots<x_n<\tilde{\psi}_{t+h,h}(y_n)).
$$
Now all assertions follow from inductive assumption and \eqref{eq20_1} similarly to the case $n=1.$

\end{proof}

\section{Example. Arratia flows with drift}

In this section we apply the developed constructions to the Arratia flow with drift. At first we recall the construction of  corresponding transition probabilities (see \cite{LeJanRaimond, DV, Riabov} for details). 

Let $a:\mathbb{R}\to\mathbb{R}$ be a Lipschitz function with the Lipschitz constant $L:$
\begin{equation}
	\label{eq14_1}
|a(x)-a(y)|\leq L|x-y|, \ x,y\in\mathbb{R}.
\end{equation}
 Consider a system of stochastic differential equations
\begin{equation}
	\label{eq14_1}
	\begin{cases}
	dX_1(t)=a(X_1(t))dt+dW_1(t), \\
	\ldots, \\
	dX_n(t)=a(X_n(t))dt+dW_n(t),
	\end{cases}
\end{equation}
where $w_1,\ldots,w_n$ are independent Wiener processes. For every initial value $x\in\mathbb{R}^n$ there is a unique strong solution of \eqref{eq14_1}. By $P^{(n),ind.}$ we denote a corresponding transition probability. Transition probabilities for the Arratia flow with drift are constructed from $\{P^{(n),ind.}:n\geq 1\}$ by coalescing finite-point motions at a meeting time. Formally this is done in the following theorem from \cite{LeJanRaimond} (see also \cite[L. 4.1]{Riabov}).

\begin{theorem}
\label{lem18_1}
\cite[Th. 4.1]{LeJanRaimond} There exists a unique compatible sequence  $\{P^{(n)}:n\geq 1\}$ of coalescing Feller transition probabilities that satisfy the following property.

Consider a starting point $x=(x_1,\ldots,x_n)\in\mathbb{R}$ and two $\mathbb{R}^n-$valued processes: $\{Y(t):t\geq 0\}$ -- a Feller process with starting point $x$ and transition probabilities $\{{P}^{(n),ind.}:t\geq 0\},$ and  $\{X(t):t\geq 0\}$ -- a Feller process with starting point $x$ and transition probabilities $\{P^{(n)}:t\geq 0\}.$ Further, let 
$$
\tau_Y=\inf\{t\geq 0:\exists i<j \ Y_i(t)=Y_j(t)\}, \tau_X=\inf\{t\geq 0:\exists i<j \ X_i(t)=X_j(t)\},
$$ 
be first meeting times for trajectories of processes $Y$ and $X,$ correspondingly. Then stopped processes $\{Y(t\wedge \tau_Y):t\geq 0\}$ and $\{X(t\wedge \tau_X):t\geq 0\}$ are identically distributed.

\end{theorem}

\begin{definition}
\label{def14_1} A stochastic flow $\psi$ is an Arratia flow with drift $a,$ if its finite-point motions are determined by transition probabilities $P^{(n)}$ from the theorem \ref{lem18_1}. 

\end{definition}

Throughout this section  $\{P^{(n)}:n\geq 1\}$  denote the sequence of transition probabilities for finite-point motions of the Arratia flow with drift $a.$ By $\mathbb{P}^{(n)}_x$ we denote the distribution in $C([0,\infty),\mathbb{R}^n)$ of $n$ trajectories from the Arratia flow with drift $a$. Properties {\bf (TP4)-(TP6)} for semigroups $P^{(n)}$ were proved  in \cite[Section 4.1]{Riabov}. In the next lemma we prove  the property {\bf (TP7)} and the condition \eqref{eq21_3} from the theorem \ref{thm1}. Recall the  function 
$$
w_{\alpha,\beta}(\varepsilon,\delta)=\inf\{t>0:\sup_{x\in[\alpha,\beta]}P^{(1)}_t(x,(x-\varepsilon,x+\varepsilon)^c)\geq \delta t\}
$$
defined in \eqref{eq08_1}. To apply the theorem \ref{thm1} we need an estimate on the asymptotic behaviour of the function $w_{\alpha,\beta}.$ We do this by comparing $w_{\alpha,\beta}$ with the function 
$$
g(x)=\sqrt{\frac{2}{\pi}}x^2\int^\infty_x e^{-\frac{z^2}{2}}dz.
$$
There exists $x_*>0$ such that $g$ is strictly increasing  on $[0,x_*]$ and strictly decreasing on $[x_*,\infty).$ 
Let 
$$
g^{-1}:(0,g(x_*)]\to [x_*,\infty)
$$
be the inverse of $g.$ Asymptotics of $g$ is well known \cite[L. 1.1.3]{Bogachev}:
$$
g(x)\sim \sqrt{\frac{2}{\pi}}x e^{-\frac{x^2}{2}}, \ x\to\infty.
$$
Consequently,
\begin{equation}
\label{eq19_2}
g^{-1}(\varepsilon)\sim \sqrt{2|\ln \varepsilon|}, \ \varepsilon\to 0.
\end{equation}
 
\begin{lemma}
\label{lem14_1} Consider arbitary $\alpha<\beta$ and $t>0.$

\begin{enumerate}

\item For any $p\in(1,\frac{3}{2})$ there is a constant $C=C(\alpha,\beta,t,p)>0$ such that for any reals  $x_1,x_2,x_3$ 
$$
\mathbb{P}^{(3)}_{x_1,x_2,x_3}(\forall s\in[0,t] \ \alpha\leq X^{(3)}_1(s)<X^{(3)}_2(s)<X^{(3)}_3(s)\leq \beta)\leq C (x_3-x_1)^{\frac{3}{p}}
$$

\item Let $M=\sup_{\alpha\leq x\leq \beta}|a(x)|.$ Assume that $\varepsilon,\delta>0$ be such that $\varepsilon^2 \delta<32 g(x_*),$ $\varepsilon <\frac{4(1+M)\log 2}{L}$ and
$$
\bigg(\frac{\varepsilon}{4g^{-1}(\frac{\varepsilon^2 \delta}{32})}\bigg)^2<\frac{\varepsilon}{4(1+M)}.
$$
Then 
$$
w_{\alpha,\beta}(\varepsilon,\delta)\geq \bigg(\frac{\varepsilon}{4g^{-1}(\frac{\varepsilon^2 \delta}{32})}\bigg)^2.
$$

\end{enumerate}

\end{lemma}

\begin{proof} 

\begin{enumerate}

\item As above, $M=\sup_{\alpha\leq x\leq \beta}|a(x)|.$ Denote the event of interest by 
$A,$
$$
A=\{f\in C([0,\infty),\mathbb{R}^3):\forall s\in[0,t] \ \  \alpha\leq f_1(s)<f_2(s)<f_3(s)\leq \beta\}.
$$
Also, let $Q_{x_1,x_2,x_3}$ be the Wiener measure, i.e. the distribution in $C([0,\infty),\mathbb{R}^3)$ of the process $w(t)=(w_1(t),w_2(t),w_3(t)),$ where $w_1,w_2,w_3$ are independent Wiener processes, $w_j(0)=x_j,$ $1\leq j\leq 3.$ By the Girsanov theorem and the H{\H o}lder inequality,
$$
\mathbb{P}^{(3)}_{x_1,x_2,x_3}(\forall s\in[0,t] \ \alpha\leq X^{(3)}_1(s)<X^{(3)}_2(s)<X^{(3)}_3(s)\leq \beta)=
$$
$$
=\mathbb{E}_{Q_{x_1,x_2,x_3}}1_A e^{\sum^3_{j=1}(\int^t_0 a(w_j(s))dw_j(s)-\frac{1}{2}\int^t_0 a(w_j(s))^2ds)}\leq 
$$
$$
\leq Q_{x_1,x_2,x_3}(A)^{\frac{1}{p}}\bigg(\mathbb{E}_{Q_{x_1,x_2,x_3}}1_A e^{\sum^3_{j=1}(\int^t_0 qa(w_j(s))dw_j(s)-\frac{q}{2}\int^t_0 a(w_j(s))^2ds)}\bigg)^{\frac{1}{q}}\leq 
$$
$$
\leq e^{\frac{3}{2}M(q-1)}Q_{x_1,x_2,x_3}(A)^{\frac{1}{p}}\bigg(\mathbb{E}_{Q_{x_1,x_2,x_3}}1_A e^{\sum^3_{j=1}(\int^t_0 qa(w_j(s))dw_j(s)-\frac{1}{2}\int^t_0 (qa(w_j(s)))^2ds)}\bigg)^{\frac{1}{q}}\leq 
$$
$$
\leq e^{\frac{3}{2}M(q-1)}Q_{x_1,x_2,x_3}(A)^{\frac{1}{p}}\leq e^{\frac{3}{2}M(q-1)} C(x_3-x_1)^{\frac{3}{p}},
$$
where the last inequality follows from \cite[Section 3]{CU}.

\item  Let $x\in [\alpha,\beta].$ Consider positive $t<(\frac{\varepsilon}{4g^{-1}(\frac{\varepsilon^2 \delta}{32})})^2.$  Denote by $\{W(t):t\geq 0\}$ a Wiener process starting from zero, and let $\{X(t):t\geq 0\}$ be a solution of the stochastic differential equation
$$
\begin{cases}
dX(t)=a(X(t))dt+dW(t) \\
X(0)=x
\end{cases}.
$$
Then $\{X(t):t\geq 0\}$ is a Feller process with initial value $X(0)=x$ and transition probability $\{P^{(1)}_t:t\geq 0\}.$ Denote $\xi=\max_{[0,t]}|W|.$ For every $s\in[0,t]$ we have 
$$
|X(s)-x|=\bigg|\int^s_0 a(X(r))dr+W(s) \bigg|\leq |W(s)|+s|a(x)|+\int^s_0|a(X(r))-a(x)|dr\leq 
$$
$$
\leq \xi +tM+L\int^s_0 |X(r)-x|dr.
$$
By Gronwall inequality,
$$
|X(t)-x|\leq (\xi+tM)e^{Lt}.
$$
From the assumptions on $\varepsilon$ and $\delta,$
$$
t<\frac{\varepsilon}{4(1+M)}<\frac{\log 2}{L}.
$$
Hence,
$$
|X(t)-x|\leq 2(\xi+tM)<2\xi +\frac{\varepsilon}{2}.
$$
It follows that 
$$
P^{(1)}_t(x,(x-\varepsilon,x+\varepsilon)^c)=\mathbb{P}(|X(t)-x|\geq \varepsilon)\leq \mathbb{P}\bigg(\xi\geq \frac{\varepsilon}{4}\bigg);
$$
$$
=2\mathbb{P}\bigg(|W(t)|\geq \frac{\varepsilon}{4}\bigg)=\frac{32 t}{\varepsilon^2}g\bigg(\frac{\varepsilon}{4\sqrt{t}}\bigg)
$$
By assumption
$$
\frac{\varepsilon}{4\sqrt{t}}>g^{-1}\bigg(\frac{\varepsilon^2 \delta}{32}\bigg).
$$
So,
$$
\frac{1}{t}\sup_{x\in[\alpha,\beta]}P^{(1)}_t(x,(x-\varepsilon,x+\varepsilon)^c)<\delta.
$$
Since the latter is true for any $t<(\frac{\varepsilon}{4g^{-1}(\frac{\varepsilon^2 \delta}{32})})^2$ we deduce that 
$$
w_{\alpha,\beta}(\varepsilon,\delta)\geq  \bigg(\frac{\varepsilon}{4g^{-1}(\frac{\varepsilon^2 \delta}{32})}\bigg)^2.
$$

\end{enumerate}
\end{proof}

\begin{corollary} 
\label{cor14_1} There exists a metric dynamical system $(\mathbb{F},\mathcal{A},\mu,(\theta_h)_{h\in\mathbb{R}}),$ a perfect cocycle $\varphi$ and a backward perfect cocycle $\tilde{\varphi},$ such that  On a suitable probability space there exist a forward stochastic flow $\psi$ and a backward stochastic flow  $\tilde{\psi}$ such that

\begin{enumerate}
\item the flow $\psi_{s,t}(\omega,x)=\varphi(t-s,\theta_s\omega,x)$ is the Arratia flow with drift $a;$

\item the backward flow $\tilde{\psi}_{t,s}(\omega,x)=\tilde{\varphi}(t-s,\theta_s\omega,x)$ is a backward Arratia flow with drift $-a;$

\item the backward stochastic flow $\tilde{\psi}$ is dual to the stochastic flow $\psi.$

\end{enumerate}

\end{corollary}

\begin{proof}
Given $\alpha<\beta$ we put $p=\frac{5}{4}$ and define $f(\varepsilon)=C\varepsilon^3,$ where $C$ is found in lemma \ref{lem14_1}.   The theorem \ref{thm1} is applicable, if we take $\varepsilon_n=2^{-n}$ and $\delta_n=\frac{1}{n}.$ Indeed,  $\varepsilon_n,\delta_n\to 0,$  and for large enough $n$ conditions of lemma \ref{lem14_1} are verified:
$$
\lim_{n\to\infty}\frac{f(8\cdot 2^{-n})}{w_{\alpha,\beta}(2^{-n},n^{-1})}\leq 2048 C \lim_{n\to\infty}2^{-n} g^{-1}\bigg(\frac{1}{32n4^n}\bigg)^2=
$$
$$
=2048C\lim_{n\to\infty}\frac{2n\log 4+2\log (32n)}{2^n}=0,
$$
where we used equivalence \eqref{eq19_2}.

We only need to check that finite-point motions of the dual flow are given by transition probabilities of the Arratia flow with the drift $-a.$ Let $p_t(x,y)$ be the transition probability for the one-point motion of the Arratia flow with drift $a,$ i.e.
$$
P^{(1)}_t(x,B)=\int_B p_t(x,y)dy.
$$
By the theorem \ref{thm3} the one-point motion of the dual flow has the transition probability 
$$
\tilde{p}_t(y,x)=-\int^y_{-\infty}\frac{\partial p_t(x,z)}{\partial x}dz.
$$
From this the equality 
$$
\frac{\partial p_t(x,y)}{\partial x}=-\frac{\partial \tilde{p}_t(y,x)}{\partial y}
$$
follows. It is then straightforward to check that 
$$
\frac{\partial \tilde{p}_t(y,x)}{\partial t}=-a(y)\frac{\partial \tilde{p}_t(y,x)}{\partial y}+\frac{1}{2}\frac{\partial^2 \tilde{p}_t(y,x)}{\partial y^2},
$$
i.e. the one-point motion $(\tilde{\psi}_{0,-t}(y))_{t\geq 0}$ of the dual flow is a weak solution of the equation
$$
\begin{cases}
d\tilde{X}(t)=-a(\tilde{X}(t))dt+dW(t) \\
\tilde{X}(0)=y
\end{cases}
$$
Independence before meeting time follows from the representation \eqref{eq21_2} and an analogous property of the forward flow.

\end{proof}

\section{Appendix}

\begin{lemma}
\label{lem_app} Let $\{P_{t}:t\geq 0\}$ be a Feller transition probability on $\mathbb{R}$ such that corresponding Feller process has continuous trajectories. Then for any $c\in\mathbb{R}$  and $t\geq 0$
$$
\lim_{x\to \infty}P_t(x,[c,\infty))=1 \mbox{ and }\lim_{x\to -\infty}P_t(x,(-\infty,c])=1
$$ 

\end{lemma}

\begin{proof} We consider the case $x\to\infty.$ By $\mathbb{P}_x$ we denote the distribution in $C([0,\infty);\mathbb{R})$ of the Feller process $\{X(t):t\geq 0\}$ with initial value $X(0)=x$ and transition probabilities $\{P_t:t\geq 0\}.$ 

Let $\varepsilon>0.$ By continuity of trajectories there exists $d<-|c|$ such that 
$$
\mathbb{P}_0 (\max_{s\in[0,t]}|X(s)|\geq |d|)\leq \varepsilon.
$$
By the Feller property, 
$$
\lim_{x\to\infty}P_t(x,(d,c))=0.
$$
Then 
$$
\limsup_{x\to\infty}P_t(x,(-\infty,c))=\limsup_{x\to\infty}P_t(x,(-\infty,d]).
$$
Let $\tau=\inf\{t\geq 0:X(t)=0\}.$  Since $d<0$ we have for all $x>0$ 
$$
P_t(x,(-\infty,d])=\mathbb{P}_x(X(t)\leq d)=\mathbb{E}_x1_{\tau< t}P_{t-\tau}(0,(-\infty,d])\leq \mathbb{P}_0 (\max_{s\in[0,t]}|X(s)|\geq |d|)\leq \varepsilon.
$$
This proves the convergence $P_t(x,(-\infty,c))\to 0,$ $x\to\infty.$

\end{proof}

\end{document}